\documentclass[12pt]{article}
\usepackage[centertags]{amsmath}
\usepackage{amsfonts}
\usepackage{amssymb}
\usepackage{latexsym}
\usepackage{amsthm}
\usepackage{newlfont}
\usepackage{graphicx}
\usepackage{listings}
\usepackage{booktabs}
\usepackage{abstract}
\date{}

\newlength{\defbaselineskip}
\newcommand{\setlinespacing}[1]%
           {\setlength{\baselineskip}{#1 \defbaselineskip}}

\newcommand{\actaqed}{\hfill $\actabox$}
{\medskip\noindent \textit{Proof of #1. }}%
{\actaqed \medskip}

\def\D{{\mathcal D}}

\def \<{\langle}
\def\>{\rangle}

\def \e{\epsilon}

\def \ff{\varphi}
\def\ga{\gamma}

\def \sp{\operatorname{span}}
\def \csp{\overline{\operatorname{span}}}

\def\bt{\beta}
\def\la{\lambda}

\newtheorem{Theorem}{Theorem}[section]
\newtheorem{Lemma}{Lemma}[section]

\newtheorem{Remark}{Remark}[section]

\numberwithin{equation}{section}

\newcommand{\be}{\begin{equation}}
\newcommand{\ee}{\end{equation}}

\begin{document}
\title{{Sparse approximation of individual functions} }
\author{ L. Burusheva\thanks{Lomonosov Moscow State University, } \,  and V. Temlyakov\thanks{University of South Carolina, Steklov Institute of Mathematics, and Lomonosov Moscow State University.} } \maketitle
\begin{abstract}
{Results on two different settings of asymptotic behavior of approximation characteristics of individual functions are presented. First, we discuss the following classical question for sparse approximation. Is it true that for any individual function from a given function class its sequence of errors of best sparse approximations with respect to a given 
 dictionary decays faster than the corresponding supremum over the function class? Second, we discuss sparse approximation by greedy type algorithms. We show that for any individual function from a given class we can improve the upper bound on the rate of convergence of the error of approximation by a greedy algorithm if we use some information from the previous iterations of the algorithm. We call bounds of this type {\it a posteriori bounds}.}
\end{abstract}

\section{Introduction}

  Asymptotic behavior of approximation characteristics of individual functions and of function classes are the most important fundamental problems of approximation theory. A discussion of interplay between approximation characteristics of a function class and an individual function from that class goes back to A. Lebesgue (1909) and S.N. Bernstein (1945) (see \cite{VT83}, Section 2.2, for a detailed discussion and references). In this paper we address the issue of 
 asymptotic behavior of nonlinear $m$-term approximation of individual elements (functions) of a Banach space. We present results on two different settings of this problem. We now give a very brief description of these problems and give a more detailed description later.
 In Sections \ref{A} and \ref{B} we discuss the following classical setting. Let $F$ be a given function class. Is it true that for any individual function $f\in F$ the best $m$-term approximation with respect to a given 
 dictionary decays faster than the corresponding supremum over the function class $F$? 
 In Section \ref{D} we discuss sparse approximation by greedy type algorithms. We show that for any individual function $f\in F$ we can improve the upper bound on the rate of convergence of the error of approximation by a given greedy algorithm if we use some information from the previous iterations of the algorithm. We call bounds of this type {\it a posteriori bounds}. 
 We now proceed to a detailed discussion.

Let $X$ be a real Banach space with norm $\|\cdot\|$. We say that a set of elements (functions) $\D$ from $X$ is a dictionary (symmetric dictionary) if each $g\in \D$ has norm bounded by one ($\|g\|\le1$),
$$
g\in \D \quad \text{implies} \quad -g \in \D,
$$
and $\csp \D =X$. Denote
$$
A_1^o(\D) :=\left\{f\in X: f=\sum_{i=1}^\infty c_ig_i, \quad g_i\in\D,\quad \sum_{i=1}^\infty |c_i|\le 1\right\}
$$
and denote by $A_1(\D)$ the closure in $X$ of $A_1^o(\D)$. We use the standard notation $\Sigma_m(\D)$ for the set of $m$-sparse with respect to $\D$ elements. 
We begin with a result on best $m$-term approximation:
$$
\sigma_m(f,\D):=\sigma_m(f,\D)_X:= \inf_{c_1,\dots,c_m; g_1,\dots,g_m, g_i\in\D}\left\|f-\sum_{i=1}^m c_ig_i\right\|_X.
$$

Our results are formulated in terms of modulus of smoothness of the Banach space $X$. For a Banach space $X$ we define the modulus of smoothness
$$
\rho(u) := \sup_{\|x\|=\|y\|=1} \frac{\|x+uy\|+\|x-uy\|}{2} - 1.
$$
The uniformly smooth Banach space is the one with the property
$$
\lim_{u\to 0}\rho(u)/u =0.
$$

It is well known (see \cite{VTbook}, Chapter 6) that for a Banach space $X$ with a power type modulus of smoothness $\rho(u)\le \ga u^q$, $1<q\le2$,  for $f\in A_1(\D)$ there exists (provided by the Relaxed Greedy Algorithm) $G_m(f)\in \Sigma_m(\D)\cap A_1^o(\D)$ such that 
\be\label{A1}
\sigma_m(f,\D)\le \|f-G_m(f)\| \le C(q,\ga)m^{-1/p},\quad p:=\frac{q}{q-1}.
\ee
We prove in Section \ref{A} the following $o$-bound. 
\begin{Theorem}\label{AT1} Let $X$ be a Banach space with a power type modulus of smoothness $\rho(u)\le \ga u^q$, $1<q\le2$. Then for any $f\in A_1^o(\D)$ we have
\be\label{A2}
\sigma_m(f,\D) =o(m^{-1/p}),\quad p:=\frac{q}{q-1}.
\ee
\end{Theorem}
 
In Section \ref{B} for each $q\in (1,2]$ we give an example of a dictionary and an element in a Banach space $\ell_q$, which shows that   we cannot replace assumption $f\in A_1^o(\D)$ by a weaker assumption $f\in A_1(\D)$ in Theorem \ref{AT1} (see Theorems \ref{BT1} and \ref{BT2}). Also, in Section \ref{B} we provide an example that shows that the $o$-bound result 
does not hold if we replace best $m$-term approximation by a greedy algorithm. Theorem \ref{BT3} gives a corresponding result for the Orthogonal Greedy Algorithm.
Concluding a discussion of the rate of approximation of individual functions from classes  $ A_1^o(\D)$ and $ A_1(\D)$, we can state that the following two interesting phenomena have been discovered. First, we established that there is the $o$-bound phenomenon for the best $m$-term approximation for the class $A_1^o(\D)$ and there is no such phenomenon for the class $ A_1(\D)$, which is the closure of $A_1^o(\D)$. Second, we established that there is no 
$o$-bound phenomenon for the Orthogonal Greedy Algorithm for the class $A_1^o(\D)$.

We now proceed to a discussion of {\it a posteriori} bounds. We illustrate our results on the example of the Weak Chebyshev Greedy Algorithm, which we define momentarily. 
Introduce a new norm, associated with a dictionary $\D$, in the dual space $X'$ by the formula
$$
\|F\|_\D:=\sup_{g\in\D}F(g),\quad F\in X'.
$$
For a nonzero element $f\in X$ we let $F_f$ denote a norming (peak) functional for $f$: 
$$
\|F_f\| =1,\qquad F_f(f) =\|f\|.
$$
The existence of such a functional is guaranteed by Hahn-Banach theorem.

Let 
$\tau := \{t_k\}_{k=1}^\infty$ be a given sequence of  nonnegative numbers $t_k \le 1$, $k=1,\dots$. We define the Weak Chebyshev Greedy Algorithm (WCGA) (see~\cite{T15} and~\cite{VTbook}, Chapter 6) that is a generalization for Banach spaces of the Weak Orthogonal Greedy Algorithm defined and studied in~\cite{VT75} (see also~\cite{DTe} for Orthogonal Greedy Algorithm).

{\bf Weak Chebyshev Greedy Algorithm (WCGA)}
We define $f^c_0 := f^{c,\tau}_0 :=f$. Then for each $m\ge 1$ we inductively define

1). $\varphi^{c}_m :=\varphi^{c,\tau}_m \in \D$ is any satisfying
\[
\label{1.1}
F_{f^{c}_{m-1}}(\varphi^{c}_m) \ge t_m  \| F_{f^{c}_{m-1}}\|_\D. 
\]

2). Define
$$
\Phi_m := \Phi^\tau_m := \sp \{\varphi^{c}_j\}_{j=1}^m,
$$
and define $G_m^c := G_m^{c,\tau}$ to be the best approximant to $f$ from $\Phi_m$.

3). Denote
$$
f^{c}_m := f^{c,\tau}_m := f-G^c_m.
$$
 The following theorem is proved in~\cite{T15} (see also \cite{VTbook}, Section 6.2). 
\begin{Theorem}\label{IT1} Let $X$ be a uniformly smooth Banach space with modulus of smoothness $\rho(u)\le \gamma u^q$, $1<q\le 2$. Take a number $\e\ge 0$ and two elements $f$, $f^\e$ from $X$ such that
$$
\|f-f^\e\| \le \e,\quad
f^\e/A(\e) \in A_1(\D),
$$
with some number $A(\e)>0$.
Then we have
\be\label{I2}
\|f^{c,\tau}_m\| \le \max\left\{2\e,\, C(q,\gamma)(A(\e)+\e)\Big(1+\sum_{k=1}^mt_k^p\Big)^{-1/p}\right\},
\quad p:=\frac{q}{q-1}.
\ee
\end{Theorem}

Theorem \ref{IT1} gives the rate of convergence of the WCGA based on the {\it a priori} information on the element $f$. One of the main goals of this paper is to improve the error bound (\ref{I2}) of Theorem \ref{IT1} using the information, which can be obtained at the
previous $m$ iterations of the WCGA. We introduce some notations. Let $P_\Phi$ denote the operator of Chebyshev projection onto subspace $\Phi$ (the operator of mapping to the best approximant in $\Phi$). Denote 
\be\label{I3}
\ff_m := \ff_m^c -P_{\Phi_{m-1}}(\ff_m^c),\quad v_m:=\|\ff_m\|.
\ee
Clearly, $v_m\le 1$ for all $m$. 
We prove in Section \ref{D} the following {\it a posteriori} result.

\begin{Theorem}\label{IT1v} Let $X$ be a uniformly smooth Banach space with modulus of smoothness $\rho(u)\le \gamma u^q$, $1<q\le 2$. Take a number $\e\ge 0$ and two elements $f$, $f^\e$ from $X$ such that
$$
\|f-f^\e\| \le \e,\quad
f^\e/A(\e) \in A_1(\D),
$$
with some number $A(\e)>0$.
Then we have
$$
\|f^{c,\tau}_m\| \le \max\left\{2\e,\, C(q,\gamma)A(\e)\Big(1+\sum_{k=1}^m(t_k/v_k)^p\Big)^{-1/p}\right\},
\quad p:=\frac{q}{q-1}.
$$
with $C(q,\ga)= 4(2\ga)^{1/q}$.
\end{Theorem}

\begin{Remark}\label{IR1} At the greedy step (step 1).) of the WCGA we have a freedom in choosing an element $\ff_m^c$. We only require the inequality
\be\label{I5}
F_{f^{c}_{m-1}}(\varphi^{c}_m) \ge t_m  \| F_{f^{c}_{m-1}}\|_\D.
\ee
Theorem \ref{IT1v} shows that we can use this freedom to our advantage, choosing at the $m$th iteration of the algorithm an element
$\ff_m^c$, which satisfies (\ref{I5}), with the smallest $v_m$. 
\end{Remark}

Note that the first result on the {\it a posteriori} error bound in a style of Theorem \ref{IT1v} was obtained in \cite{GQTC}. The authors proved Theorem \ref{IT1v} in a special case of a Hilbert space under assumption that $\e=0$ (in other words, under assumption $f\in A_1(\D)$).

\section{Upper bounds in approximation of individual elements}
\label{A}

 {\bf Proof of Theorem \ref{AT1}.}  Theorem \ref{AT1} is a direct corollary of the following lemma. 
\begin{Lemma}\label{AL1} Let $X$ be a Banach space with a power type modulus of smoothness $\rho(u)\le \ga u^q$, $1<q\le2$. Then for any $f\in A_1^o(\D)$ there exist a sequence of numbers
$\{\delta_m\}$ such that $\delta_m \to 0$ as $m\to \infty$ and a sequence of elements $s_m\in \Sigma_{2m}(\D)$ with properties: $s_m\in A_1^o(\D)$ and
\be\label{A3}
\|f-s_m\|_X \le \delta_m m^{-1/p}.
\ee
\end{Lemma}
\begin{proof} Let $f\in A_1^o(\D)$ have a representation 
$$
f=\sum_{i=1}^\infty c_ig_i,\quad g_i\in \D,\quad \sum_{i=1}^\infty |c_i| \le 1.
$$
Denote
$$
\bt_m:= \sum_{i>m}|c_i|.
$$
Clearly, $\bt_m\to 0$ as $m\to\infty$. We build the approximant $s_m$ as a sum of two $m$-term approximants: $s_m=s_m^1+s_m^2$. Set
$$
s_m^1 := \sum_{i=1}^m c_ig_i.
$$
Denote $h_m:= (f-s_m^1)\bt_m^{-1}$. It is clear that $h_m\in A_1^o(\D)$. By (\ref{A1})
\be\label{A4}
\|h_m-G_m(h_m)\|\le C(q,\ga)m^{-1/p}.
\ee
Set 
$$
s_m^2 := \bt_mG_m(h_m).
$$
Then
$$
\|f-s_m\| \le \bt_m\|h_m-G_m(h_m)\| \le \bt_mC(q,\ga)m^{-1/p}.
$$
It is clear that $s_m\in A_1^o(\D)$. Thus, setting $\delta_m:=\bt_m C(q,\ga)$ we complete the proof of Lemma \ref{AL1}.
\end{proof}

 \section{Lower bounds in approximation of individual elements}
\label{B}
We begin with a result for the Hilbert space $\ell_2$.
\begin{Theorem}\label{BT1}
There are a dictionary $\D$ in the Hilbert space $\ell_2$ and an element $f\in A_1(\D)$ such that 
$$
\sigma _m(f,\D)=\frac{1}{\sqrt{2}} (m+1)^{-1/2}.
$$
\end{Theorem}
\begin{proof} Let $\{e_k\}_{k=1}^\infty$ be a canonical basis of $\ell_2$, i.e. $e_1=(1,0,0,\dots)$, $e_k=(0,\dots,0,1,0,\dots)$ with $1$ at the $k$th place, $k=2,3,\dots$. 
Let us take  a dictionary $\D:=\{g_k\}_{k=1}^\infty$ such that $g_k:=1/\sqrt{2}(e_{k+1}-e_1)$, $k=1,2,\dots$.  
It is clear that $\|g_k\|=1$, $k=1,2,\dots$. It will follow from our further argument that the closure 
of $\D$ is the whole space $\ell_2$. For that it is sufficient to check that $e_1$ can be approximated arbitrarily well by linear combinations of elements from $\D$. Consider $f:=-\frac{1}{\sqrt{2}}e_1$. Then it is clear that
$$
\sigma_m(f,\D)=\inf_{c_1,\dots,c_m; \varphi_1,\dots,\varphi_m, \varphi_i\in\D}\left\|f-\sum_{i=1}^m c_i\varphi_i\right\|=\inf_{c_1,\dots,c_m}\left\|f-\sum_{i=1}^m c_ig_i\right\|.
$$

Therefore,  $\sigma_m(f,\D)$ is equal to the distance from $f$ to the $m$-dimensional subspace spanned by $g_1,...,g_m$. It is easy to see that the element $h:=1/(m+1)\sum_{i=1}^m g_i$ is the orthogonal projection of the vector $f$ onto that subspace. It follows from the identity
$$
f-h = -\frac{1}{\sqrt{2}(m+1)}\sum_{k=1}^{m+1} e_k.
$$ 
Therefore,
$$
\sigma_m(f,\D)=\left\|f- h\right\|=\sqrt{(m+1)\cdot\frac{1}{2(m+1)^2}}=\frac{1}{\sqrt{2(m+1)}}.
$$
In particular, this implies that $\D$ is a dictionary and that $f\in A_1(\D)$.
 The proof of Theorem \ref{BT1} is complete.
\end{proof}

We now prove a result similar to Theorem \ref{BT1} for the Banach spaces $\ell_q$, $1<q<2$.

\begin{Theorem}\label{BT2} Let $1<q<2$.
There are a dictionary $\D$ in the Banach space $\ell_q$ and an element $f\in A_1(\D)$ such that 
$$
\sigma _m(f,\D)\ge  2^{-1-1/q} m^{-1/p},\quad p=\frac{q}{q-1}.
$$
\end{Theorem}
\begin{proof} Let $\{e_k\}_{k=1}^\infty$ be a canonical basis of $\ell_q$, i.e. $e_1=(1,0,0,\dots)$, $e_k=(0,\dots,0,1,0,\dots)$ with $1$ at the $k$th place, $k=2,3,\dots$. 
Let us take  a dictionary $\D:=\{g_k\}_{k=1}^\infty$ such that $g_k:= 2^{-1/q}(e_{k+1}-e_1)$, $k=1,2,\dots$.  
It is clear that $\|g_k\|:=\|g_k\|_{\ell_q}=1$, $k=1,2,\dots$.   Consider $f:=- 2^{-1/q}e_1$. Then it is clear that
\be\label{B1}
\sigma_m(f,\D)=\inf_{c_1,\dots,c_m; \varphi_1,\dots,\varphi_m, \varphi_i\in\D}\left\|f-\sum_{i=1}^m c_i\varphi_i\right\|=\inf_{c_1,\dots,c_m}\left\|f-\sum_{i=1}^m c_ig_i\right\|.
\ee

Therefore,  $\sigma_m(f,\D)$ is equal to the distance from $f$ to the $m$-dimensional subspace spanned by $g_1,...,g_m$. First, we prove that $\D$ is a dictionary and that $f\in A_1(\D)$.
Consider the element $h:=1/(m+1)\sum_{i=1}^m g_i$.   It follows from the identity
$$
f-h = -\frac{1}{ 2^{1/q}(m+1)}\sum_{k=1}^{m+1} e_k
$$ 
that 
$$
\left\|f- h\right\|=\left((m+1)\cdot\frac{1}{2(m+1)^q}\right)^{1/q}\le 2^{-1/q} m^{-1/p}.
$$
This implies that $\D$ is a dictionary and that $f\in A_1(\D)$.

Second, we estimate from below the $\sigma_m(f,\D)$. We have from (\ref{B1})
$$
 \sigma_m(f,\D)=\inf_{c_1,\dots,c_m}\left\|f-\sum_{i=1}^m c_ig_i\right\| = \inf_{c_1,\dots,c_m}2^{-1/q}\left(\left|1-\sum_{k=1}^m c_k\right|^q +\sum_{k=1}^m |c_k|^q\right)^{1/q}.
 $$
 Therefore, if $\sum_{k=1}^m c_k \le 1/2$ then $\sigma_m(f,\D) \ge 2^{-1-1/q}$. If $\sum_{k=1}^m c_k \ge 1/2$ then
 $$
 \left(\sum_{k=1}^m |c_k|^q\right)^{1/q}\ge m^{-1/p}\sum_{k=1}^m |c_k| \ge 2^{-1}m^{-1/p}.
 $$
 
 This completes the proof of Theorem \ref{BT2}.
  
\end{proof}
Note, that a little refinement of the above proof of Theorem \ref{BT2} gives the asymptotic 
relation
\be\label{B2}
\sigma_m(f,\D)\approx 2^{-1/q} m^{-1/p}.
\ee

It is well known that the space $\ell_q$, $1<q\le2$, is a uniformly smooth Banach space with modulus of smoothness $\rho(u)\le \gamma u^q$. Thus, Theorems \ref{BT1} and \ref{BT2} show that for all $q\in (1,2]$ the condition $f\in A_1^o(\D)$ in Theorem \ref{AT1} cannot be replaced by a weaker condition $f\in A_1(\D)$.

Theorem \ref{AT1} provides the $o$-bound phenomenon for the best $m$-term approximation of elements from $A_1^o(\D)$. We now show on the example of the Orthogonal Greedy Algorithm (OGA), which is the WCGA with weakness sequence $\tau =\{1\}$ defined on a Hilbert space, that there is no $o$-bound phenomenon for greedy algorithms. 

\begin{Theorem}\label{BT3}
There are a dictionary $\D$ in the Hilbert space $\ell_2$ and an element $f\in A_1^o(\D)$, which is a linear combination of two elements of the dictionary $\D$, such that for the $m$th residual $f_m^o$ of one of the realizations of the OGA we have
$$
 \|f_m^o\|=\frac{1}{\sqrt{2}} (m+1)^{-1/2}.
$$
\end{Theorem}
\begin{proof} Let, as above, $\{e_k\}_{k=1}^\infty$ be a canonical basis of $\ell_2$. 
Take  a dictionary $\D:=\{g_k\}_{k=1}^\infty\cup \{u_1,u_2\}$ such that $g_k:=1/\sqrt{2}(e_{k+2}-e_2)$, $k=1,2,\dots$ and 
$$
u_1:= \frac{1}{\sqrt{2}}(e_1-e_2),\quad u_2:= \frac{1}{\sqrt{2}}(-e_1-e_2).
$$
It is clear that  $\D$ is a dictionary for the space $\ell_2$.    Consider $f:=-\frac{1}{\sqrt{2}}e_2= (u_1+u_2)/2 \in A_1^o(\D)$. Then it is clear that at the first iteration of the OGA we can choose $g_1$. By induction we can show that one of the realizations of the OGA consists in choosing element $g_n$ at the $n$th iteration of the OGA. Indeed, suppose that after $n$ iterations we have chosen $g_1,\dots,g_n$. Then in the same way as in the proof of Theorem \ref{BT1} we 
obtain that the element $h:=1/(n+1)\sum_{i=1}^n g_i$  is the orthogonal projection of the vector $f$ onto the subspace spanned by $g_1,...,g_n$ and
$$
f_n^o=f-h = -\frac{1}{ \sqrt{2}(n+1)}\sum_{k=2}^{n+2} e_k.
$$ 
Thus, we have for all $k>n$ that 
$$
|\<f_n^o,g_k\>|= |\<f_n^o,u_1\>|=|\<f_n^o,u_2\>|
$$
and, therefore, we can choose $g_{n+1}$ at the $(n+1)$th iteration of the OGA. 
This implies that
$$
\|f_m^o\|=\frac{1}{\sqrt{2(m+1)}}.
$$
 The proof of Theorem \ref{BT3} is complete.
\end{proof}

\begin{Remark}\label{BR1} We can make a slight modification of two elements of the dictionary
 in the proof of Theorem \ref{BT3} 
$$
u_1=(1/\sqrt{2}+\delta',-1/\sqrt{2}+\delta,0,0,...),
$$
$$
u_2=(-1/\sqrt{2}-\delta',-1/\sqrt{2}+\delta,0,0,...)
$$
where $\delta\in(0,1/\sqrt{2})$ and $||h_1||=||h_2||=1$ to guarantee that for the element $f:= (u_1+u_2)/2$ we have
$$
||f_m^o||=\left(\frac{1}{\sqrt{2}} -\delta\right) (m+1)^{-1/2}
$$
for all realizations of the OGA.

\end{Remark}

We note that a result similar to Theorem \ref{BT3} can be derived from \cite{TZ} (see also the proof of Theorem 5.24 in \cite{VTbook}, pp. 304--305). However, the above direct proof of Theorem \ref{BT3} is technically less involved. 

We now demonstrate that the technique used in the proof of Theorem \ref{BT3} can be used in proving a negative result for other type of greedy algorithm, namely, for the Relaxed Greedy Algorithm (RGA). We begin with the definition of the RGA (see, for instance, \cite{VTbook}, pp. 82--83). Let $H$ be a real Hilbert space and $\D$ be a dictionary in $H$. Consider a symmetrized dictionary $\D^\pm := \{\pm g: \, g\in\D\}$. For an element $h\in H$, let $g(h)$ denote an element from $\D^\pm$ which maximizes $\<h,g\>$ over all element $g\in\D^\pm$ (we assume the existence of such an element).

{\bf Relaxed Greedy Algorithm (RGA).} Let $f\in H$. Denote $f_0^r:=f$, $G_0^r(f):=0$. Then, for each $m\ge 1$ we inductively define
$$
G_m^r(f):=\left(1-\frac{1}{m+1}\right)G_{m-1}^r(f) + \frac{1}{m+1}g(f_{m-1}^r),
$$
$$
f_m^r:=f-G_m^r(f).
$$

\begin{Theorem}\label{BT4}
There are a dictionary $\D$ in the Hilbert space $\ell_2$ and an element $f\in A_1^o(\D)$, which is a linear combination of two elements of the dictionary $\D$, such that for the $m$th residual $f_m^r$ of one of the realizations of the RGA we have
$$
\|f_m^r\|=\frac{1}{\sqrt{2}} (m+1)^{-1/2}.
$$
\end{Theorem}
\begin{proof} Let $\D:=\{g_k\}_{k=1}^\infty\cup \{u_1,u_2\}$ and  $f:=-\frac{1}{\sqrt{2}}e_2= (u_1+u_2)/2 \in A_1^o(\D)$ be from the proof of Theorem \ref{BT3}. Then it is clear that at the first iteration of the RGA we can choose $g_1$. By induction we can show that one of the realizations of the RGA consists in choosing element $g_n$ at the $n$th iteration of the RGA. Indeed, suppose that after $n$ iterations we have chosen $g_1,\dots,g_n$. Then in the same way as in the proof of Theorem \ref{BT1}  for the RGA we have 
$$
f_n^r=f-G_n^r(f)=f-\frac{1}{n+1}\sum_{j=1}^n g_j=-\frac{1}{ \sqrt{2}(n+1)}\sum_{j=2}^{n+2} e_j.
$$
Hence as above for the OGA 
  we can choose $g_{n+1}$ at the $(n+1)$th iteration of the RGA. 
It follows that
$$
\|f_m^r\|=\frac{1}{\sqrt{2(m+1)}}.
$$

 The proof of Theorem \ref{BT4} is complete.
\end{proof}

\section{A posteriori error bounds in Banach spaces}
\label{D}

We begin with a proof of Theorem \ref{IT1v} from the Introduction. 

{\bf Proof of Theorem \ref{IT1v}.} The proof is based on the following analog of the General Error Reduction Lemma from \cite{DeTe}. 

\begin{Lemma}\label{DL1} {\bf General Error Reduction Lemma.}  Let $X$ be a uniformly smooth Banach space with modulus of smoothness $\rho(u)$. Take a number $\e\ge 0$ and two elements $f$, $f^\e$ from $X$ such that
$$
\|f-f^\e\| \le \e,\quad
f^\e/A(\e) \in A_1(\D),
$$
with some number $A(\e)>0$.

Suppose that $f$ is represented $f=f'+G'$ in such a way that $F_{f'}(G')=0$ and an element 
$\ff'$, $\|\ff'\|=1$, is chosen to satisfy $F_{f'}(\ff') \ge \theta \|F_{f'}\|_\D$, $\theta \ge 0$. 
Then we have  
\be\label{D1}
\inf_{\la\ge 0}\|f'-\la \ff'\| \le \|f'\|\inf_{\la\ge0}\left(1-\la \theta A(\e)^{-1}\left(1-\frac{\e}{\|f'\|}\right) + 2\rho\left(\frac{\la}{\|f'\|}\right)\right).
\ee
\end{Lemma}

\begin{proof}
  For any $\la\ge 0$ we have 
\[
\|f'-\la\ff'\| + \|f'+\la\ff'\| \le 2\|f'\|\left(1+\rho(\la/\|f'\|)\right).
\]
Next,
\[
\|f' +\la \ff'\| \ge F_{f'}(f' +\la \ff') = \|f'\|+\la F_{f'}(\ff').
\]
By our assumption we get
$$
F_{f'}(\ff') \ge \theta \sup_{g\in \D}F_{f'}(g).
$$
Using Lemma 6.9 from \cite{VTbook}, p.343, and our assumption on $f$ and $f^\e$ we continue
\[
=\theta \sup_{\phi\in A_1(\D)}F_{f'}(\phi) \ge \theta A(\e)^{-1}F_{f'}(f^\e) \ge 
\theta A(\e)^{-1}(F_{f'}(f)-\e).
\]
The property $F_{f'}(G')=0$ provides 
\be\label{3.8n}
F_{f'}(f) = F_{f'}(f' + G') = F_{f'}(f') = \|f'\|.
\ee
Combining the above relations   we complete the proof of Lemma~\ref{DL1}.

\end{proof}

Consider the $k$th iteration of the WCGA. Set $f':= f_{k-1}^c$ and $G':=G_{k-1}^c$. Then, it is well known that $F_{f_{k-1}^c}(G_{k-1}^c) =0$ (see, for instance, \cite{VTbook}, Lemma 6.9, p.342). Thus, the condition $F_{f'}(G')=0$ is satisfied. Next, choose
$$
\ff_k:= \ff_k^c -P^c_{\Phi_{k-1}}(\ff_k^c),\qquad \ff':=\ff_k/v_k . 
$$
Further,
$$
F_{f'}(\ff') = F_{f^c_{k-1}}(\ff_k/v_k) = F_{f^c_{k-1}}(\ff_k^c/v_k) \ge (t_k/v_k) \|F_{f^c_{k-1}}\|_\D.
$$

Then by Lemma \ref{DL1} we obtain
$$
\|f_k^c\| \le \inf_{\la\ge 0}\|f_{k-1}^c-\la \ff'\|
$$
\be\label{D2}
  \le \|f_{k-1}^c\|\inf_{\la\ge0}\left(1-\frac{\la (t_k/v_k)}{ A(\e)}\left(1-\frac{\e}{\|f_{k-1}^c\|}\right) + 2\rho\left(\frac{\la}{\|f_{k-1}^c\|}\right)\right).
\ee

We continue the proof of Theorem~\ref{IT1v}. It is clear that it suffices to consider the case $A(\e)\ge\e$. Otherwise, $\|f_m^c\|\le\|f\|\le\|f^\e\|+\e\le2\e$. 
Also, assume $\|f_m^c\|>2\e$ (otherwise, Theorem~\ref{IT1v} trivially holds). Then, by monotonicity of $\{\|f_k^c\|\}$ we have for all $k=0,1,\dots,m$ that $\|f_k^c\|>2\e$. Set $a_k:= \|f_k^c\|$. Inequality (\ref{D2}) gives
\be\label{D3}
a_k  \le a_{k-1}\inf_{\la\ge0}\left(1-\frac{\la (t_k/v_k)}{ 2A(\e)} + 2\rho\left(\frac{\la}{a_{k-1}}\right)\right).
\ee
We complete estimation of $a_m$ by application of Lemma \ref{EL2}. We need to specify the corresponding parameters from Lemma \ref{EL2} and check that its conditions are satisfied. 
Set $B:= 2A(\e)$. Then inequality (\ref{D3}) gives inequality (\ref{E1}) from Lemma \ref{EL2} with $r_k=t_k/v_k$. By our assumption $A(\e)\ge \e$ we have
$$
a_0:= \|f\| \le \|f^\e\|+\e \le A(\e)+\e \le 2A(\e) =B.
$$
Finally, from the definition of modulus of smoothness $\rho(u)$ it follows that $\rho(2)\ge 1$. This implies $\ga 2^q\ge 1$. Therefore, applying Lemma \ref{EL2} we complete the proof of Theorem \ref{IT1v}.

We now proceed to the Weak Greedy Algorithm with Free Relaxation. The following version of relaxed greedy algorithm was introduced and studied in~\cite{T26} (see also~\cite{VTbook}, Chapter 6).
 
{\bf Weak Greedy Algorithm with Free Relaxation  (WGAFR).}
Let $\tau:=\{t_m\}_{m=1}^\infty$, $t_m\in[0,1]$, be a weakness  sequence. We define $f_0^e   :=f$ and $G_0^e  := 0$. Then for each $m\ge 1$ we inductively define

1). $\varphi_m^e   \in \D$ is any satisfying
$$
F_{f_{m-1}^e}(\varphi_m^e  ) \ge t_m  \|F_{f_{m-1}^e}\|_\D.
$$

2). Find $w_m$ and $ \lambda_m $ such that
$$
\|f-((1-w_m)G_{m-1}^e + \la_m\varphi_m^e)\| = \inf_{ \la,w}\|f-((1-w)G_{m-1}^e + \la\varphi_m^e)\|
$$
and define
$$
G_m^e:=   (1-w_m)G_{m-1}^e + \la_m\varphi_m^e.
$$

3). Denote
$$
f_m^e   := f-G_m^e.
$$

In a spirit, the WGAFR is close to the WCGA. The greedy steps (steps 1).) are identical. The approximation steps (steps 2).) are similar. In the WCGA we define $G_m^c$ to be the Chebyshev projection of $f$ on the at most $m$-dimensional $\Phi_m$ and in the WGAFR we define $G_m^e = P^c_{\Phi_m^e}(f)$, $\Phi_m^e := \sp(\ff_m^e, G_{m-1}^e)$, to be the Chebyshev projection of $f$ on the at most two-dimensional subspace $\Phi_m^e$. It is known
that an analog of Theorem \ref{IT1} holds for the WGAFR as well (see, for instance, \cite{VTbook}, Theorem 6.23, p.353). Here we formulate an analog of Theorem \ref{IT1v} for the WGAFR. Denote
$$
\phi_m := \ff_m^e - P^c_{\Phi_{m-1}^e}(\ff_m^e),\qquad u_m:=\|\phi_m\|.
$$
Then the following {\it a posteriori} result holds for the WGAFR.

\begin{Theorem}\label{DT1} Let $X$ be a uniformly smooth Banach space with modulus of smoothness $\rho(u)\le \gamma u^q$, $1<q\le 2$. Take a number $\e\ge 0$ and two elements $f$, $f^\e$ from $X$ such that
$$
\|f-f^\e\| \le \e,\quad
f^\e/A(\e) \in A_1(\D),
$$
with some number $A(\e)>0$.
Then we have
$$
\|f^{e}_m\| \le \max\left\{2\e,\, C(q,\gamma)A(\e)\Big(1+\sum_{k=1}^m(t_k/u_k)^p\Big)^{-1/p}\right\},
\quad p:=\frac{q}{q-1}.
$$
with $C(q,\ga)= 4(2\ga)^{1/q}$.
\end{Theorem}

\begin{Remark}\label{DR1} At the greedy step (step 1).) of the WGAFR we have a freedom in choosing an element $\ff_m^e$. We only require the inequality
\be\label{D5}
F_{f^{e}_{m-1}}(\varphi^{e}_m) \ge t_m  \| F_{f^{e}_{m-1}}\|_\D.
\ee
Theorem \ref{DT1} shows that we can use this freedom to our advantage, choosing at the $m$th iteration of the algorithm an element
$\ff_m^e$, which satisfies (\ref{D5}), with the smallest $u_m$. 
\end{Remark}

The proof of Theorem \ref{DT1} repeats the proof of Theorem \ref{IT1v}. We do not present it here. 

\section{Some technical lemmas}
\label{E}

We begin with a simple known lemma from \cite{VT75} (see Lemma 3.1 there). For the reader's convenience we present a proof of this lemma, which goes along the lines of the proof of Lemma 2.16 from \cite{VTbook}, p.91.

\begin{Lemma}\label{EL1} Let $\{a_k\}_{k=0}^m$ be a sequence of nonnegative numbers satisfying the inequalities
$$
a_0 \le A,\qquad a_k\le a_{k-1}(1-r_ka_{k-1}/A),\quad k=1,2,\dots,m
$$
with some nonnegative $r_k$. Then we have
$$
a_m\le A\left( 1+\sum_{k=1}^m r_k\right)^{-1}.
$$
\end{Lemma}
\begin{proof} Our assumption implies that $a_0\ge a_1 \ge \cdots \ge a_m$. If $a_m=0$ then the conclusion of Lemma \ref{EL1} is trivial. Suppose that $a_m>0$. Then for all $k\le m$ we have $a_k>0$ and
$$
\frac{1}{a_m} \ge \frac{1}{a_{m-1}}\frac{1}{1-r_ma_{m-1}/A} \ge  \frac{1}{a_{m-1}}\left(1+r_ma_{m-1}/A\right) 
$$
$$
\ge  \frac{1}{a_{m-1}} + \frac{r_m}{A} \ge\cdots\ge  \frac{1}{a_{0}} + \frac{1}{A}\sum_{k=1}^m r_m \ge
\frac{1}{A} \left(1+\sum_{k=1}^m r_m\right),
$$
which proves the lemma.
\end{proof}

The following lemma is often used in an implicit form in proofs of the rate of approximation of greedy algorithms in Banach spaces (see, for instance, \cite{VTbook}, p.345). 

\begin{Lemma}\label{EL2} Let sequences $\{a_k\}_{k=0}^m$ of positive numbers, 
$\{r_k\}_{k=1}^m$ of nonnegative numbers, and numbers $B>0$, $\ga>0$, $q\in (1,2]$ be such 
that $a_0\le B$, for $k=1,\dots,m$
\be\label{E1}
a_k \le a_{k-1}\inf_{\la\ge 0} \left(1-\frac{\la r_k}{B} +2\ga\left( \frac{\la}{a_{k-1}}\right)^q\right),
\ee
and, in addition, $\ga 2^q \ge 1$. Then 
$$
a_m \le C(q,\ga)B\left(1+\sum_{k=1}^m r_k^p\right)^{-1/p},\quad p:=\frac{q}{q-1},
$$
with $C(q,\ga) = 2(2\ga)^{1/q}$.
\end{Lemma}
\begin{proof}
 Our assumptions guarantee that $B\ge a_0 \ge a_1 \ge \cdots \ge a_m$. 
 Choose $\la$ from the equation
$$
\frac{\la r_k}{2B} = 2\gamma \left(\frac{\la}{a_{k-1}}\right)^q
$$
what implies that
$$
\la = a_{k-1}^{\frac{q}{q-1}} (4\gamma B)^{-\frac{1}{q-1}}r_k^{\frac{1}{q-1}}.
$$
Denote 
$$
A_q := 2(4\gamma)^{\frac{1}{q-1}} .
$$
Using notation $p:= \frac{q}{q-1}$ we get from (\ref{E1})
$$
a_k \le a_{k-1}\left(1-\frac{1}{2}\frac{\la r_k}{B}\right) = a_{k-1}\left(1-\frac{r_k^pa_{k-1}^p}{A_qB^p}\right).
$$
Raising both sides of this inequality to the power $p$ and taking into account the inequality $x^r\le x$ for $r\ge 1$, $0\le x\le 1$, we obtain
$$
a_k^p \le a_{k-1}^p \left(1-\frac{r^p_ka_{k-1}^p}{A_qB^p}\right).
$$
By Lemma \ref{EL1}, using the bounds $a_0 \le B$ and $A_q>1$, we get
$$
a_m^p \le A_qB^p\left(1+\sum_{k=1}^m r_k^p\right)^{-1},
$$
which implies
$$
a_m\le C(q,\gamma)B\left(1+\sum_{k=1}^m r_k^p\right)^{-1/p}
$$
with $C(q,\gamma)=A_q^{1/p}= 2(2\gamma)^{1/q}$.
Lemma \ref{EL2} is proved. 

 \end{proof}
 {\bf Acknowledgement.} The work was supported by the Russian Federation Government Grant N{\textsuperscript{\underline{o}}}14.W03.31.0031.


\begin{thebibliography}{9999}

\bibitem{DeTe} A. Dereventsov and V.N. Temlyakov, A unified way of analyzing some greedy algorithms, arXiv:1801.06198v1 [math.NA] 18 Jan 2018.

\bibitem{DTe} R.A. DeVore and V.N. Temlyakov, Some remarks on Greedy Algorithms, Advances in Computational Mathematics {\bf 5} (1996), 173--187.

\bibitem{GQTC} You Gao, Tao Qian, Vladimir Temlyakov, Long-fei Cao, Aspects of 2D-Adaptive Fourier
Decompositions, arXiv:1710.09277v1 [math.NA] 24 Oct 2017. 
   
\bibitem{T15}  V.N. Temlyakov, Greedy algorithms in Banach spaces,
 Adv. Comput. Math., {\bf 14} (2001),  277--292.
 
\bibitem{VT83} V.N. Temlyakov, Nonlinear Methods of Approximation, Found. Comput. Math., {\bf 3} (2003), 33--107.
  
\bibitem{T26}  V.N. Temlyakov, Relaxation in greedy approximation,   Constructive Approximation, {\bf 28} (2008), 1--25.

\bibitem{VT75} V.N. Temlyakov, Weak Greedy Algorithms, Advances in Comp. Math., {\bf 12} (2000), 213--227.

\bibitem{VTbook} V.N. Temlyakov, Greedy approximation, Cambridge University Press, 2011.

 \bibitem{TZ} V.N. Temlyakov and P. Zheltov, On performance of greedy algorithms,
J. Approximation Theory, 2011, Vol. 163, 1134--1145.


\end{thebibliography}
\end{document}